\documentclass[11pt]{amsart}
\usepackage{amsmath,amssymb}
\usepackage{graphicx}

\usepackage{fullpage}
\usepackage{enumerate}
\usepackage{pgf,tikz}
\usetikzlibrary{arrows}
\usetikzlibrary{positioning}
\usetikzlibrary{patterns}

\def\COMMENT#1{}
\def\TASK#1{}

\newdimen\margin   
\def\textno#1&#2\par{%
    \margin=\hsize
    \advance\margin by -4\parindent
           \setbox1=\hbox{\sl#1}%
    \ifdim\wd1 < \margin
       $$\box1\eqno#2$$%
    \else
       \bigbreak
       \hbox to \hsize{\indent$\vcenter{\advance\hsize by -3\parindent
       \sl\noindent#1}\hfil#2$}%
       \bigbreak
    \fi}

\newtheorem{thm}{Theorem}[section]
\newtheorem{define}[thm]{Definition}

\newtheorem{lem}[thm]{Lemma}

\newtheorem{prop}[thm]{Proposition}

\newtheorem{question}[thm]{Question}

\newtheorem*{thm*}{Theorem}
\newtheorem*{define*}{Definition}
\newtheorem*{examp*}{Example}
\newtheorem*{lem*}{Lemma}
\newtheorem*{claim*}{Claim}
\newtheorem*{fact*}{Fact}
\newtheorem*{col*}{Corollary}
\newtheorem*{conj*}{Conjecture}

\begin{document}

\title{On deficiency problems for graphs}

\author{Andrea Freschi, Joseph Hyde and Andrew Treglown}

\date{}
\begin{abstract}
Motivated by analogous questions in the setting of Steiner triple systems and Latin squares, Nenadov, Sudakov and Wagner [Completion and deficiency problems, Journal of Combinatorial Theory Series B, 2020]
recently introduced the notion of \emph{graph deficiency}. Given a global spanning property $\mathcal P$ and a graph $G$, the deficiency $\text{def}(G)$ of the graph $G$ with respect to the property $\mathcal P$ is the smallest non-negative integer $t$ such that the join $G*K_t$ has property $\mathcal P$.
In particular, Nenadov, Sudakov and Wagner raised the question of determining how many edges an $n$-vertex graph $G$ needs to ensure $G*K_t$ contains a $K_r$-factor (for any fixed $r\geq 3$). 
In this paper we resolve their problem fully. We also give an analogous result which forces $G*K_t$ to contain any fixed bipartite $(n+t)$-vertex graph of bounded degree and small bandwidth.
\end{abstract}
\maketitle

\section{Introduction} \label{Introduction}
A natural question dating back to the 1970s asks for the order of the smallest complete Steiner triple system a fixed  partial Steiner triple system can be embedded into~(see e.g.~\cite{bh, ccl, nsw}). 
Similarly, there has been interest in establishing the order of the smallest Latin square that a fixed partial Latin square can be embedded into~(see e.g.~\cite{dh, evans, nsw}).

Motivated by these research directions, Nenadov, Sudakov and Wagner~\cite{nsw} introduced the notion of
\emph{graph deficiency}: for a graph $G$ and integer $t \geq 0$, denote by $G*K_t$ the \emph{join} of $G$ and $K_t$, which is the graph obtained from $G$ by adding $t$ new vertices and adding all edges incident to at least one of the new vertices.
Given  a global spanning property $\mathcal P$ and a graph $G$, the {\emph{deficiency}} $\text{def}(G)$ of the graph $G$ with respect to the property $\mathcal P$ is the smallest $t\geq 0$ such that the join $G*K_t$ has property $\mathcal P$.

Note that the following special type of deficiency problem has been previously studied: given a graph $H$ and  $n \in \mathbb N$, what is the minimum number of vertices needed to ensure any  $H$-packing on $n$ vertices (i.e., a collection of edge-disjoint copies of $H$ that together form a graph on  $n$ vertices) can be extended to an $H$-design (i.e., an $H$-packing of a complete graph)? See e.g.~\cite{furl, fur2} for background and results on this problem.

One of the main results in~\cite{nsw} is a bound on $\text{def}(G)$ with respect to the Hamiltonicity property for graphs $G$ of a given density.
More precisely, the following result answers the question of how many edges an $n$-vertex graph $G$ can have such that $G*K_t$ does not contain a Hamilton cycle.

\begin{thm}[Nenadov, Sudakov and Wagner~\cite{nsw}]\label{nswcycle}
Let $n$ and $t$ be integers and $G$ an $n$-vertex graph so that $G*K_t$ does not contain a Hamilton cycle. Then we have the following bounds on $e(G)$.
\begin{itemize}
\item If $n+t$ is even:
$$e(G) \leq \binom{n}{2} -
\begin{cases}
\left(t(n-1)-\binom{t}{2}\right) & \text{if $t\leq (n+4)/5$}\\[0.5em]
\left(\binom{\frac{n+t+2}{2}}{2}-1\right) & \text{if $t\geq (n+4)/5$.}
\end{cases}$$

\item If $n+t$ is odd:
$$e(G) \leq \binom{n}{2} -
\begin{cases}
\left(t(n-1)-\binom{t}{2}\right) & \text{if $t\leq (n+1)/5$}\\[0.5em]
\binom{\frac{n+t+1}{2}}{2} & \text{if $t\geq (n+1)/5$.}
\end{cases}$$

\end{itemize}
These bounds on $e(G)$ are sharp.
\end{thm}

Another line of inquiry in~\cite{nsw} concerns  the deficiency problem for $K_r$-factors.
Given graphs $H$ and $G$, an \emph{$H$-factor} in $G$ is a collection of vertex-disjoint copies of $H$ in $G$ that together cover all the vertices of $G$.
Note that $H$-factors are also often referred to as \emph{perfect $H$-tilings}, \emph{perfect $H$-packings} or \emph{perfect $H$-matchings}.
The following seminal result of Hajnal and Szemer\'edi~\cite{hs} determines the minimum degree threshold for forcing a $K_r$-factor in a graph $G$.

\begin{thm}[Hajnal and Szemer\'edi~\cite{hs}]\label{hs}
Every graph $G$ on $n$ vertices with $r|n$ and whose minimum degree satisfies $\delta (G) \geq (1-1/r)n$ contains a  $K_r$-factor. Moreover, there are $n$-vertex graphs $G$
 with $\delta (G) = (1-1/r)n-1$ that do not contain a $K_r$-factor.
\end{thm}
More recently, K\"uhn and Osthus~\cite{kuhn2} determined, up to an additive constant, the minimum degree threshold for forcing an $H$-factor, for any fixed graph $H$.

The following result of Nenadov, Sudakov and Wagner~\cite{nsw} determines how many edges an $n$-vertex graph $G$ needs to guarantee that $G*K_t$ contains a $K_3$-factor (provided that $t$ is not too big compared to $n$).

\begin{thm}[Nenadov, Sudakov and Wagner~\cite{nsw}]\label{nswtriangle}
There exists $n_0 \in \mathbb{N}$ such that the following holds. 
Let $n,t \in \mathbb N$ so that $n \geq n_0$ and $3|(n+t)$, and let $G$ be an $n$-vertex graph such that $G*K_t$ does not contain a $K_3$-factor. If $t\leq n/1000$ then
$$e(G)\leq \binom{n}{2}-\binom{k}{2}-\begin{cases}
k(n-k) & \text{if $t$ is odd}\\
k(n-k-1) & \text{if $t$ is even,}
\end{cases} $$
where $k:=\lceil (t+1)/2 \rceil$. This bound on $e(G)$ is sharp.
\end{thm}

\subsection{A deficiency result for $K_r$-factors}
Nenadov, Sudakov and Wagner~\cite{nsw} state that {\it{`the study of deficiency concept by itself leads to intriguing open problems'}}. 
In particular, the first open problem~\cite[Section~7]{nsw} they raise is to extend Theorem~\ref{nswtriangle} to the full range of $t$, and moreover to resolve the analogous question for $K_r$-factors in general. 
In this paper we fully resolve this problem via the following theorem.

\begin{thm}\label{mainthm}
Let $n,t,r \in \mathbb{Z}$ with $n\geq 2$, $t\geq 0$
and $r \geq 3$ such that $t < (r-1)n$ and $r |(n+t)$. Further, let $k := \lceil\frac{t+1}{r-1}\rceil$ and $q$ be the integer remainder when $t$ is divided by $r-1$. 
Let $G$ be a graph on $n $ vertices such that $G*K_t$ does not contain a $K_r$-factor. Then $$e(G)\leq\max\left\{\binom{n}{2} - \binom{\frac{n+t}{r} + 1}{2}, \binom{n}{2} - \binom{k}{2} - k(n - k - (r - 2 - q))\right\}.$$
\end{thm}

When $(r-1)|(t+1)$, the first term is at most the second term precisely when $t \leq \frac{(r-1)n-r^2}{2r^2-2r+1}$.
Note that Theorem~\ref{mainthm} considers all interesting values of $n$ and $t$. Indeed, if $t \geq (r-1)n$ and $r|(n+t)$, then $G*K_t$ trivially contains a $K_r$-factor, even if $e(G)=0$. 
Further, in Section~\ref{exexsec} we provide extremal examples that demonstrate that the edge condition in Theorem~\ref{mainthm} cannot be lowered.
Perhaps surprisingly, the proof of Theorem~\ref{mainthm} is short, making use of a couple of vertex-modification tricks (see the proofs of Lemmas~\ref{vertexlemma} and \ref{edgelemma}) and Theorem~\ref{hs}.

Note that the $t=0$ case of Theorem~\ref{mainthm} determines the edge density threshold for forcing a $K_r$-factor in a graph. In fact, this is an old result due to Akiyama and Frankl~\cite{af}, so our result can be viewed as a deficiency generalisation of their theorem.

\subsection{A deficiency bandwidth theorem} 
One of the central results in extremal graph theory is the 
so-called \emph{Bandwidth  theorem} due to B\"ottcher, Schacht and Taraz~\cite{bot}.
A graph $H$ on $n$ vertices is said to have \emph{bandwidth at most $b$}, if there exists a 
labelling of the vertices of $H$ with the numbers $1, \dots ,n$ such that for every edge $ij\in E(H)$ we have $|i-j|\leq b$.
\begin{thm}[The Bandwidth  theorem, B\"ottcher, Schacht and Taraz~\cite{bot}]\label{bst}
Given any $r,\Delta \in \mathbb N$ and any $\gamma >0$, there exist constants $\beta >0$ and $n_0 \in \mathbb N$
such that the following holds. Suppose that $H$ is an $r$-chromatic graph on $n \geq n_0$ vertices with $\Delta (H) \leq \Delta$ 
and bandwidth at most $\beta n$. If $G$ is a graph on $n$ vertices with
$$\delta (G) \geq \left(1- \frac{1}{r}+\gamma \right)n,$$
then $G$ contains a copy of $H$. 
\end{thm}
Note that a $K_r$-factor has bandwidth $r-1$;
thus, one can view the bandwidth theorem as a vast asymptotic generalisation of  Theorem~\ref{hs}. 

Following the proof of Theorem~\ref{nswcycle} from~\cite{nsw}, and applying a theorem of Knox and the third author \cite{knox}, one can easily obtain
a deficiency result for embedding bipartite graphs of bounded degree and small bandwidth.

\begin{thm}\label{bandwidththm}
Given any $\Delta \in \mathbb{N}$ and $\varepsilon > 0$, there exist constants $\beta > 0$ and $n_0 \in \mathbb{N}$ such that the following holds. Let $t \in \mathbb N $ and $n\geq n_0$.
Let $H$ be a bipartite graph on $n + t $ vertices with $\Delta (H) \leq \Delta$ 
and bandwidth at most $\beta(n+t)$. Suppose that $G$ is a graph on $n$ vertices such that $G*K_t$ does not contain a copy of $H$. Then we have the following bound on $e(G)$.

 \begin{align*}
         e(G) \leq   \binom{n}{2} -     \begin{cases}
                                            \left(t(n-1) - \binom{t}{2} - \varepsilon n^2\right) \ \ & \mbox{if} \ \ t \leq \frac{n}{5} \\[0.5em]
                                            \left(\binom{\lceil\frac{n+t}{2}\rceil + 1}{2} - \varepsilon n^2\right) \ \ & \mbox{if} \ \ t > \frac{n}{5}.
                                        \end{cases}
\end{align*}
\end{thm}

Observe that the bounds on $e(G)$ in Theorem~\ref{bandwidththm} are, up to error terms, exactly the same as those in Theorem~\ref{nswcycle}. Moreover, the extremal examples that show the condition on $e(G)$ in Theorem~\ref{nswcycle}
is sharp also demonstrate that, for many graphs $H$, the condition on $e(G)$ in
 Theorem~\ref{bandwidththm} is asymptotically best possible (see Section~\ref{exexsec}).
Notice the statement of Theorem~\ref{bandwidththm} is only interesting for $t< n-1$. 
Indeed, if $t\geq n-1$ then even if $G$ has no edges, $G*K_t$ contains all $(n+t)$-vertex bipartite
graphs $H$ (we just embed the smallest colour class into $K_t$).

\smallskip 

The paper is organised as follows.
We introduce some graph theoretic notation in Section~\ref{notationsec}. In Section~\ref{exexsec} we provide extremal examples for Theorems~\ref{mainthm} and \ref{bandwidththm}. In Section~\ref{mainthmsec} we prove Theorem~\ref{mainthm} and then in Section~\ref{proofofbandwidththmsection} we prove Theorem~\ref{bandwidththm}. Some concluding remarks are given in Section~\ref{conclusionsec}.

\section{Notation}\label{notationsec}

Let \(G\) be a graph. We define \(V(G)\) to be the vertex set of \(G\) and \(E(G)\) to be the edge set of \(G\). Let \(X \subseteq V(G)\). Then \(G[X]\) is the \textit{graph induced by \(X\) on \(G\)} and has vertex set \(X\) and edge set \(E(G[X]) := \{xy \in E(G): x,y \in X\}\). For each \(x \in V(G)\), we define the \textit{neighbourhood of \(x\) in \(G\)} to be \(N_G(x):= \{y\in V(G): xy \in E(G)\}\) and define \(d_G(x) := |N_G(x)|\). 

We write \(0 < a \ll b \ll c < 1\) to mean that we can choose the constants \(a,b,c\) from right to left. More precisely, there exist non-decreasing functions \(f: (0,1] \to (0,1]\) and \(g: (0,1] \to (0,1]\) such that for all \(a \leq f(b)\) and \(b \leq g(c)\) our calculations and arguments in our proofs are correct. Larger hierarchies are defined similarly.

\section{The extremal constructions for Theorem~\ref{mainthm} and Theorem~\ref{bandwidththm}}\label{exexsec}

In this section we will give the extremal constructions that match the upper bounds in Theorems~\ref{mainthm} and \ref{bandwidththm}. Firstly, let us consider those for Theorem~\ref{mainthm}.

\begin{define}\label{exexdef1.4}
Let $n,r,t\in\mathbb{Z}$ with $r \geq 3$, $t \geq 0$ and $n \geq 2$ such that $t < (r-1)n$ and $r |(n+t)$. Further, let $k :=  \lceil\frac{t+1}{r-1}\rceil$ and $q$ be the integer remainder when $t$ is divided by $r-1$. We define graphs $EX_1(n,t,r)$ and $EX_2(n,t,r)$ as follows:
\begin{itemize}
\item Let $K := K_n$ and $A\subseteq K$ such that $A = K_{\frac{n+t}{r} + 1}$. Define $EX_1(n,t,r)$ to be the graph obtained by removing $E(A)$ from $K$.
\item Consider a set of isolated vertices $B$ where $|B| =  k$ and $K_{n-k}$, and let $C \subseteq V(K_{n-k})$ where $|C| = r-2-q$.
Define $EX_2(n,t,r)$ to be the graph obtained by taking the disjoint union of $B$ and $K_{n-k}$ and adding every edge incident to a vertex in $C$.
\end{itemize}
\end{define}

Observe that \[e(EX_1(n,t,r)) = \binom{n}{2} - \binom{\frac{n+t}{r} + 1}{2}\] and \[e(EX_2(n,t,r)) = \binom{n}{2} - \binom{k}{2} - k(n - k - (r - 2 - q)).\]
Hence  
\begin{small}\[\max\left\{\binom{n}{2} - \binom{\frac{n+t}{r} + 1}{2}, \binom{n}{2} - \binom{k}{2} - k(n - k - (r - 2 - q))\right\}  = \max\left\{e(EX_1(n,t,r)), e(EX_2(n,t,r))\right\}.\]\end{small}

Next we show that $EX_1(n,t,r)*K_t$ and $EX_2(n,t,r)*K_t$ do not contain $K_r$-factors, that is, they are extremal graphs for Theorem~\ref{mainthm}.

\begin{prop}
 $EX_1(n,t,r)*K_t$ and $EX_2(n,t,r)*K_t$ do not contain $K_r$-factors.
\end{prop}

\begin{proof}
Firstly, let us consider $EX_1(n,t,r)*K_t$. For a contradiction, assume that $EX_1(n,t,r)*K_t$ contains a $K_r$-factor $\mathcal T$. 
Then each vertex in $V(A)$ belongs to a different copy of $K_r$ in $\mathcal T$. This implies $\frac{n+t}{r} = |\mathcal T| \geq |V(A)| = \frac{n+t}{r} + 1$, a contradiction. 
Hence $EX_1(n,t,r)*K_t$ does not contain a $K_r$-factor. 

Now let us consider $EX_2(n,t,r)*K_t$. For a contradiction, assume that $EX_2(n,t,r)*K_t$ contains a $K_r$-factor $\mathcal T$.
Then every vertex of $B$ belongs to a different copy of $K_r$ in $\mathcal T$. Thus, by construction, the copies of $K_r$ in $\mathcal T$ covering $B$ must cover at least 
\begin{align*}
    &  \left\lceil\frac{t+1}{r-1}\right\rceil \cdot (r-1)   
    =   \left(\frac{t-q}{r-1} + 1\right) \cdot (r-1) =  t - q + r - 1 > t + |C|
\end{align*}
vertices in the copy of $K_t$ and $C$, a contradiction. Hence $EX_2(n,t,r)*K_t$ does not contain  a $K_r$-factor.
\end{proof}

We now give the extremal constructions which, excluding error terms, match the upper bounds given in Theorem~\ref{bandwidththm}.

\begin{define}\label{exexdef1.6}
Let $n, t \in \mathbb{N}$ such that $\lceil\frac{n+t}{2}\rceil < n$. We define graphs $EX_1(n,t)$ and $EX_2(n,t)$ as follows: 
\begin{itemize}
    \item Let $K := K_n$ and $A \subseteq K$ such that $A = K_{\lceil\frac{n+t}{2}\rceil + 1}$. Define $EX_1(n,t)$ to be the graph obtained by removing $E(A)$ from $K$. 
    \item Define $EX_2(n,t)$ to be the disjoint union of a set of $t$ isolated vertices and a clique of size $n - t$.
\end{itemize}
\end{define}

One can see that the extremal examples in Definition~\ref{exexdef1.6} have the same construction to those in Definition~\ref{exexdef1.4} for $r = 2$, except that in Definition~\ref{exexdef1.6} we omit the condition $2|(n+t)$ and add the condition that $\lceil\frac{n+t}{2}\rceil < n$ (in order for $EX_1(n,t)$ to be well-defined). 

Observe that $e(EX_1(n,t))$ and $e(EX_2(n,t))$ asymptotically match the upper bounds given in Theorem~\ref{bandwidththm}. Indeed, \[e(EX_1(n,t)) = \binom{n}{2} - \binom{\lceil\frac{n+t}{2}\rceil + 1}{2} \ \ \text{  and  } \ \  e(EX_2(n,t)) = \binom{n}{2} - \left(t(n-1) - \binom{t}{2}\right).\] 

We conclude this section by showing that $EX_1(n,t)*K_t$ and $EX_2(n,t)*K_t$ do not contain certain $(n+t)$-vertex bipartite graphs $H$. 

\begin{define}
Let $\mathcal{H}_1$ be the class of bipartite graphs $H$ on $n+t$ vertices with largest independent set of size $\lceil\frac{n+t}{2}\rceil$.
Let $\mathcal{H}_2$ be the class of bipartite graphs $H$ on $n+t$ vertices which do not have a tripartition $(A,B,C)$ of $V(H)$ such that $|A| = n-t$, $|B| = |C| = t$ and every vertex in $C$ is only adjacent to vertices in $B$. 
\end{define}

For example, the Hamilton cycle (when $n+t$ is even), the disjoint union of an isolated vertex and a cycle on $n+t-1$ vertices (when $n+t$ is odd) and $K_{s,s}$-factors (for any fixed $s \in \mathbb N$) belong to $\mathcal{H}_1$;  the Hamilton cycle (when $n+t$ is even), $K_{s,s}$-factors (for any fixed $s \in \mathbb N$ so that $s$ does not divide $t$) and any bipartite graph with minimum degree at least $t+1$ belong to $\mathcal{H}_2$.

\begin{prop}
$EX_1(n,t)*K_t$ does not contain any graph in $\mathcal{H}_1$ and $EX_2(n,t)*K_t$ does not contain any graph in $\mathcal{H}_2$. 
\end{prop}

\begin{proof}
Firstly, let us consider $EX_1(n,t)*K_t$. Since $EX_1(n,t)$ contains an independent set of size $\lceil\frac{n+t}{2}\rceil + 1$, any bipartite graph from $\mathcal{H}_1$ cannot be in $EX_1(n,t)$. 

Now let us consider $EX_2(n,t)*K_t$. Since $EX_2(n,t)$ has a set of $t$ isolated vertices and $|K_t| = t$, we require that any bipartite graph $H$ spanning $EX_2(n,t)*K_t$ must have a tripartition $(A,B,C)$ of $V(H)$ such that $|A| = n-t$, $|B| = |C| = t$ and every vertex in $C$ is only adjacent to vertices in $B$, where $B = V(K_t)$ and $C$ is the set of $t$ isolated vertices in $EX_2(n,t)$. Hence $EX_2(n,t)*K_t$ does not contain any graph in $\mathcal{H}_2$. 
\end{proof}

\section{Proof of Theorem~\ref{mainthm}}\label{mainthmsec}

The proof of Theorem~\ref{mainthm} follows an inductive argument on the number of vertices $n$ of $G$. Given a graph $G$ such that $G*K_t$ does not contain a $K_r$-factor, we apply an appropriate vertex-modification procedure which, roughly speaking,
allows us to assume $G$ is locally isomorphic to one of the two extremal examples. This allows us to remove such local structure from $G$ and apply induction.

The vertex-modification procedures are described by the following two structural lemmas regarding graphs $G$ with the property that $G*K_t$ does not contain a $K_r$-factor. 
Lemma~\ref{vertexlemma} allows us to assume that the degree of a vertex is either $n-1$ (which is the degree of all vertices in $V(EX_1(n,t,r))\setminus A$ and $C\subset V(EX_2(n,t,r))$) or at most $n-1-\lceil\frac{t+1}{r-1}\rceil$ (which is the degree of all vertices in $V(EX_2(n,t,r))\setminus(B\cup C)$).
Lemma~\ref{edgelemma} allows us to assume that either each edge of $G$ belongs to some $r$-clique  or there is a vertex with degree $n-1$.

\begin{lem}\label{vertexlemma}
Let $t \geq 0$, $r\geq 3$ 
and $G$ be a graph on $n$ vertices such that $e(G)$ is maximal with respect to the property that $G*K_t$ does not contain a $K_r$-factor.
Then for every vertex $v\in V(G)$ either $d_G(v)=n-1$ or $d_G(v)\leq n-1-\lceil\frac{t+1}{r-1}\rceil$. 
\end{lem}

\begin{proof}
Suppose there exists a vertex $v\in V(G)$ such that 
$$n-1-\left\lceil\frac{t+1}{r-1}\right\rceil<d_G(v)<n-1.$$
Let $G'$ be the graph obtained from $G$ by adding every possible edge incident to $v$, that is, $d_{G'}(v) = n-1$. 
Since $e(G)$ is maximal with respect to $G*K_t$ not containing a $K_r$-factor, we must have that $G'*K_t$ contains a $K_r$-factor $\mathcal{T'}$. 
Using $\mathcal{T'}$, we will now construct a $K_r$-factor $\mathcal{T}$ in $G*K_t$, giving us a contradiction.

Let $K^v$ be the copy of $K_r$ in $\mathcal{T'}$ covering $v$. If $K^v \subseteq G*K_t$ then we can take $\mathcal{T}: = \mathcal{T'}$. 
Hence assume $K^v \nsubseteq G*K_t$. If there exists a copy $K'$ of $K_r$ in $\mathcal{T'}$ that lies entirely in $K_t$, then, for any $u \in V(K')$, we can take 
\[\mathcal{T} := \left(\mathcal{T'}\setminus\{K^v, K'\}\right) \cup \{(G*K_t)[\{u\} \cup \left(V(K^v)\setminus \{v\}\right)], (G*K_t)[\{v\} \cup \left(V(K')\setminus \{u\}\right)]\}.\] 
Hence assume no such copy of $K_r$ in $\mathcal{T'}$ exists. 
Since $K^v \nsubseteq G*K_t$ and no copy of $K_r$ in $\mathcal T'$ lies entirely in $K_t$, the number of copies of $K_r$ in $\mathcal{T'}$ which cover some vertex of $\{v\} \cup V(K_t)$ in $G'*K_t$ is at least $\lceil\frac{t+1}{r-1}\rceil$.
But $d_G(v) > n-1 - \lceil\frac{t+1}{r-1}\rceil$, hence there exists a copy $\hat{K}$ of $K_r$ in $\mathcal{T'}$ which intersects $\{v\} \cup V(K_t)$ and whose vertices are all neighbours of $v$ in $G*K_t$ or $v$ itself. 
Now, if $v \in V(\hat{K})$, then $K^v=\hat{K} \subseteq G*K_t$, a contradiction to our previous assumption. 
Hence $V(\hat{K}) \cap V(K_t) \neq \emptyset$ and, for any $u \in V(\hat{K}) \cap V(K_t)$, we can take  \[\mathcal{T} := (\mathcal{T'}\setminus\{K^v, \hat{K}\}) \cup \{(G*K_t)[\{u\} \cup \left(V(K^v)\setminus \{v\}\right)] , (G*K_t)[\{v\} \cup (V(\hat{K})\setminus \{u\})]\}.\] 
\end{proof}

\begin{lem}\label{edgelemma}
Let $t \geq 0$, $r\geq 3$.
Let $G$ be a graph on $n$ vertices such that $G*K_t$ does not contain a $K_r$-factor and suppose $G$ contains an edge which is not contained in any copy of $K_r$ in $G$. Then there exists a graph $G'$ on $n$ vertices such that $G'*K_t$ does not contain a $K_r$-factor, $e(G)\leq e(G')$ and $G'$ has a vertex of degree $n-1$. 
\end{lem}

\begin{proof}
Let $xy$ be an edge in $G$ which is not contained in any copy of $K_r$. Let $Q$ be a clique of maximal size containing $xy$ and set $\ell := |V(Q)|$. Observe that every vertex in $G$ has at most $\ell-1$ neighbours in $Q$ as otherwise $xy$ would lie in an $(\ell+1)$-clique. Thus
$$\sum_{v\in V(Q)} d_G(v) \leq n(\ell-1).$$
Let $G'$ be the graph obtained by deleting all edges between $x$ and vertices in $V(G)\setminus V(Q)$ and, subsequently, adding any missing edge incident to any vertex in $V(Q)\setminus\{x\}$. Note that $Q$ is still an $\ell$-clique in $G'$ and
$$\sum_{v\in V(Q)} d_{G'}(v) = n(\ell-1).$$
Hence $e(G)\leq e(G')$. Also, $G'$ is a graph on $n$ vertices and $d_{G'}(y) = n-1$. It remains to show that $G'*K_t$ does not contain a $K_r$-factor.
Suppose, for a contradiction, that $G'*K_t$ does contain a $K_r$-factor $\mathcal{T'}$. We will now use $\mathcal{T}'$ to construct a new $K_r$-factor $\mathcal{T}$ in $G'*K_t$ which does not contain any edge $vw$ where $v \in V(Q)$ and $w \in V(G')\setminus V(Q)$. Such a $K_r$-factor $\mathcal{T}$ is also a $K_r$-factor in $G*K_t$, giving us a contradiction.

Suppose that the copy $K^x$ of $K_r$ in $\mathcal{T'}$ covering $x$ has exactly $s$ vertices in $K_t$. Then $K^x$ has exactly $r-s$ vertices in $Q$. Thus there are $\ell-r+s$ remaining vertices in $V(Q)\setminus V(K^x)$.
Note that $\ell-r+s\leq s$, hence there is an injection $f : V(Q)\setminus V(K^x) \to V(K^x)\cap V(K_t)$. We construct our $K_r$-factor $\mathcal{T}$ as follows: for every copy of $K_r$ in $\mathcal{T}'$ intersecting $Q$ other than $K^x$, we substitute all vertices lying in its intersection with $Q$ with their images under $f$. 
Finally, we take the copy of $K_r$ formed by the $\ell$-clique $Q$ and the $s - (\ell - r + s) = r-\ell$ vertices in $V(K^x)\cap V(K_t)$ which do not appear in the image of $f$. Observe that $\mathcal{T}$ does not use any edge $vw$ where $v \in V(Q)$ and $ w \in V(G)\setminus V(Q)$, and we are done.
\end{proof}

Before proceeding to the  proof of Theorem~\ref{mainthm}, we prove the following technical lemma.

\begin{lem}\label{technicallemma}
Let $n,t,r \in \mathbb N$ such that the following holds: $n,r \geq 3$; $r-1$ divides $t+1$; $r$ divides $n+t$; $t+1<(r-1)(n-1)$. If
\begin{equation}\label{condition}
e(EX_1(n-1,t+1,r)) < e(EX_2(n-1,t+1,r))\footnote{Note that $EX_1(n-1,t+1,r)$, $EX_2(n-1,t+1,r)$ and $EX_2(n,t,r)$ are well-defined since the assumptions in Definition~\ref{exexdef1.4} are satisfied.}
\end{equation}
then
\begin{equation}\label{claim}
e(EX_2(n-1,t+1,r)) + (n-1) \leq e(EX_2(n,t,r)).
\end{equation}
\end{lem}

\begin{proof}
Let $k:=\frac{t+1}{r-1}$. Since $r-1$ divides $t+1$, we can compute $e(EX_2(n,t,r))$ and $e(EX_2(n-1,t+1,r))$ explicitly:
$$e(EX_2(n,t,r)) = {n\choose 2} - {k\choose 2} - k(n-k),$$
$$e(EX_2(n-1,t+1,r)) = {n-1\choose 2} - {k+1\choose 2} - (k+1)(n-k-r).$$
It follows that 
$$e(EX_2(n,t,r))- e(EX_2(n-1,t+1,r))=(n-1) + k +(n-k-r)-kr.$$
Rearranging this, one obtains that  \eqref{claim} holds precisely when 
\begin{equation*}\label{gnreq}
    t \leq n - r - \frac{n}{r} =: g(n,r).
\end{equation*}
Further, one can calculate that $e(EX_2(n-1,t+1,r)) \leq e(EX_1(n-1,t+1,r))$ precisely when \begin{equation}\label{f1f2eq}
    f_1(n,r) := \frac{n(r-1)}{(2r^2 - 2r + 1)} - r \leq t \leq n(r-1) - r^2 =: f_2(n,r).
\end{equation}
Since \eqref{condition} holds, we have that \eqref{f1f2eq} implies that $t < f_1(n,r)$ or $t > f_2(n,r)$. Observe that $f_1(n,r) \leq g(n,r)$.
Thus, if $t < f_1(n,r)$ then $t < g(n,r)$ and the claim holds. 

Suppose $t > f_2(n,r)$.
We will show that in this case the hypothesis of the lemma cannot actually hold. In particular,
under the assumptions that $r-1$ divides $t+1$, $r$ divides $n+t$ and $t > f_2(n,r)$,
$EX_2(n-1,t+1,r)$ is undefined. Indeed, for a contradiction let us assume that 
$EX_2(n-1,t+1,r)$ is well-defined in this case, thus the inequality $t+1<(r-1)(n-1)$ must hold. By assumption, $t$ satisfies the following modular equations:
\begin{equation} \label{congr}
    t\equiv -1 \quad(\text{mod}\; r-1)\quad\text{and}\quad t\equiv -n\quad(\text{mod}\; r).
\end{equation} 
For $n$ and $r$ fixed, the solution of \eqref{congr} is unique modulo $r(r-1)$ by the Chinese Remainder Theorem, since $r$ and $r-1$ are coprime. Note that $t'=(r-1)(n-1)-1$ is a solution of \eqref{congr}, hence $t=(r-1)(n-1)-1-kr(r-1)$ for some integer $k$. The constraint $t+1<(r-1)(n-1)$ forces $k\geq 1$, hence
\begin{equation*}
t\leq (r-1)(n-1)-1-r(r-1)=n(r-1)-r^2=f_2(n,r).
\end{equation*}
This contradicts the assumption that $t > f_2(n,r)$.
\end{proof}

\begin{proofofmainthm}
We prove Theorem~\ref{mainthm} by induction on $n$. If $n=2$ then $e(G) \in \{0, 1\}$. Since $t < 2(r-1)$ and $r |(n+t)$, we must have $t = r-2$. If $e(G) = 1$ then $G*K_t$ is a copy of $K_r$, contradicting our choice of $G$. Thus $e(G) = 0$. Observe that $k = \lceil\frac{t+1}{r-1}\rceil = 1$ and $q = r-2$. Thus \[\max\left\{\binom{n}{2} - \binom{\frac{n+t}{r} + 1}{2}, \binom{n}{2} - \binom{k}{2} - k(n - k - (r - 2 - q))\right\} = \max\{0,0\} = 0.\] Thus Theorem~\ref{mainthm} holds for $n = 2$.

For the inductive step, we may 
assume without loss of generality that $G$ is a graph on $n$ vertices such that $e(G)$ is maximal with respect to the property that $G*K_t$ does not contain a $K_r$-factor.

\smallskip

{\it Case (i): $G$ contains an isolated vertex $v$.}
 If $t < r-2$, then one could add a single edge to $v$ and $G$ would still not contain a $K_r$-factor, contradicting our choice of $G$. Hence $t \geq r-2$. If $t = r - 2$, then $k = \lceil\frac{t+1}{r-1}\rceil= 1$ and $q = r-2$. Hence 

\begin{align*}
   &\max\left\{\binom{n}{2} - \binom{\frac{n+t}{r} + 1}{2}, \binom{n}{2} - \binom{k}{2} - k(n - k - (r - 2 - q))\right\} \\= &\max\left\{\binom{n}{2} - \binom{\frac{n-2}{r} + 2}{2}, \binom{n-1}{2}\right\}.
\end{align*} Since $G$ contains at least one isolated vertex, \[e(G) \leq \binom{n-1}{2} \leq \max\left\{\binom{n}{2} - \binom{\frac{n-2}{r} + 2}{2}, \binom{n-1}{2}\right\},\] and we are done. 
If $t \geq r-1$, consider the graph $G'$ obtained by deleting $v$ from $G$. Since $G*K_t$ does not contain a $K_r$-factor, $G'*K_{t-r+1}$ does not contain a $K_r$-factor. Thus, by our inductive hypothesis
$$e(G')\leq\max\{e(EX_1(n-1,t-r+1,r)), e(EX_2(n-1,t-r+1,r))\}.$$
It follows from $e(G)=e(G')$ and $e(EX_i(n-1,t-r+1,r))\leq e(EX_i(n,t,r))$\footnote{This holds since $EX_i(n-1,t-r+1,r)$ can be obtained by removing an appropriate vertex from $EX_i(n,t,r)$, for $i=1,2$.} for $i=1,2$ that $e(G)\leq\max\{e(EX_1(n,t,r)), e(EX_2(n,t,r))\}$. 

\smallskip

{\it Case (ii): $G$ contains a vertex of degree $n-1$.}
Consider the graph $G'$ obtained by deleting such a vertex from $G$. Note that $G*K_t=G'*K_{t+1}$.
If $t+1\geq (r-1)(n-1)$, then trivially $G'*K_{t+1}$ contains a $K_r$-factor, a contradiction.
So we must have that $t+1< (r-1)(n-1)$ and
hence by induction
\begin{equation*}\label{starteq}
    e(G')\leq\max\{e(EX_1(n-1,t+1,r)), e(EX_2(n-1,t+1,r))\}.
\end{equation*} 
Observe that $e(G)=e(G')+n-1$.  We aim to show that 
\begin{equation}\label{inductioneq}
    e(G)\leq\max\{e(EX_1(n,t,r)), e(EX_2(n,t,r))\}.
\end{equation} 
If $e(G')\leq e(EX_1(n-1,t+1,r)$ then
\begin{equation*}
e(G*K_t) = e(G'*K_{t+1}) \leq e(EX_1(n-1,t+1,r)*K_{t+1}) = e(EX_1(n,t,r)*K_t)
\end{equation*}
and thus
\begin{equation*}
e(G) \leq e(EX_1(n,t,r)).
\end{equation*}
Similarly, if $e(G')\leq e(EX_2(n-1,t+1,r)$ and $r-1$ does not divide $t+1$ then
\begin{equation*}
e(G*K_t) = e(G'*K_{t+1}) \leq e(EX_2(n-1,t+1,r)*K_{t+1}) = e(EX_2(n,t,r)*K_t)
\end{equation*}
and thus
\begin{equation*}
e(G) \leq e(EX_2(n,t,r)).
\end{equation*}
It remains to check that  \eqref{inductioneq} holds in the case when 
$(r-1)|(t+1)$ and
\begin{equation*}\label{ex2eq}
    e(EX_1(n-1,t+1,r)) < e(EX_2(n-1,t+1,r)).
\end{equation*}
In this case Lemma~\ref{technicallemma} implies that
\begin{equation*} 
e(G)=e(G')+ n - 1 \leq e(EX_2(n-1,t+1,r)) + (n-1) \leq e(EX_2(n,t,r)),
\end{equation*}
as desired.

\smallskip

{\it Case (iii): $G$ contains no  vertex  of degree $0$ or $n-1$.}
If $G$ contains an edge which is not contained in any copy of $K_r$, then by Lemma~\ref{edgelemma} there exists a graph $G'$ on $n$ vertices such that $G'*K_t$ does not contain a $K_r$-factor, $e(G)\leq e(G')$ and $G'$ has a vertex of degree $n-1$. The argument from Case (ii) then implies that $e(G) \leq e(G') \leq \max\{e(EX_1(n,t,r)), e(EX_2(n,t,r))\}$.

We may therefore assume every edge of $G$ is contained in some copy of $K_r$. Moreover, as no vertex in $G$ has degree $0$, every vertex in $G$ is contained in a copy of $K_r$.
Let $w$ be a vertex of smallest degree in $G$. 
If $d_{G}(w) \geq n-\frac{n+t}{r}$ then 
$$\delta(G*K_t) \geq n-\frac{n+t}{r}+t=\frac{r-1}{r}(n+t).$$
Hence by Theorem~\ref{hs}, we have that $G*K_t$ contains a $K_r$-factor, a contradiction.

Thus $d_{G}(w)<n-\frac{n+t}{r}$. Let $K$ be a copy of $K_r$ in $G$ containing $w$. Consider the graph $G'$ obtained by removing $K$ and all its vertices from $G$.
Then $G'*K_t$ does not contain a $K_r$-factor; this implies that $t<(r-1)(n-r)$. Hence, by our inductive hypothesis, we have
$$e(G')\leq\max\{e(EX_1(n-r,t,r)), e(EX_2(n-r,t,r))\}.$$

If $e(G')\leq e(EX_1(n-r,t,r))$ then $d_{G}(w)<n-\frac{n+t}{r}$ implies $e(G)\leq e(EX_1(n,t,r))$, as desired: this follows from the fact that $EX_1(n-r,t,r)$ can be obtained by removing a clique $Q$ of size $r$ from $EX_1(n,t,r)$ where $Q$ has one vertex of degree $n-\frac{n+t}{r}-1$.
By Lemma~\ref{vertexlemma}, we have that $\Delta(G)\leq n-1-\lceil\frac{t+1}{r-1}\rceil$ since we assumed $G$ contains no vertices of degree $n-1$ and $e(G)$ is maximal with respect to the property that $G*K_t$ does not contain a $K_r$-factor. Thus, if $e(G')\leq e(EX_2(n-r,t,r))$ then applying this observation yields that $e(G)\leq e(EX_2(n,t,r))$, as desired: similarly as before, this follows from the fact that $EX_2(n-r,t,r)$ can be obtained by removing a clique $Q$ of size $r$ from $EX_2(n,t,r)$ where all vertices in $Q$ have degree $n-1-\lceil\frac{t+1}{r-1}\rceil$. \qed
\end{proofofmainthm}

\section{Proof of Theorem~\ref{bandwidththm}}\label{proofofbandwidththmsection}

In the proof of Theorem~\ref{bandwidththm} we will make use of the following theorem of Knox and Treglown~\cite{knox}.

\begin{thm}[Knox and Treglown~\cite{knox}]\label{knoxthm}
Given any $\Delta \in \mathbb N$ and any $\gamma >0$, there exists constants $\beta >0$ and $n_0 \in \mathbb N$
such that the following holds. Suppose that $H$ is a bipartite graph on $n \geq n_0$ vertices with $\Delta (H) \leq \Delta$ 
and bandwidth at most $\beta n$. Let $G$ be a graph on $n$ vertices with degree sequence $d_1\leq \dots \leq d_n$. If
$$d_{i} \geq i+\gamma n \ \text{ for all } \ i < n/2$$
then $G$ contains a copy of $H$.
\end{thm}

In fact, Knox and Treglown proved a more general result for robust expanders (see \cite[Theorem~1.8]{knox}). We use Theorem~\ref{knoxthm} in a similar way to how Chv\'{a}tal's theorem~\cite{chvatal} is used in the proof of Theorem~\ref{nswcycle} in~\cite{nsw}.

\begin{proof}
Let $\Delta \in \mathbb{N}$ and $\varepsilon > 0$. Define $\gamma > 0$ such that $\gamma \ll \varepsilon, 1/\Delta$. Apply Theorem~\ref{knoxthm} with $\Delta$ and $\gamma$ to produce constants $\beta > 0$ and $n_0 \in \mathbb{N}$ such that
\begin{equation*}\label{hierarchy}
    0 < \frac{1}{n_0} \ll \beta \ll \gamma \ll \varepsilon, \frac{1}{\Delta}.
\end{equation*} 

Let  $t \in \mathbb N$, $n \geq n_0$ and $H$ be an $(n+t)$-vertex graph as in the statement of the theorem.
Suppose that $G$ is a graph on $n $ vertices such that $G*K_t$ does not contain $H$. 
Let $m(G)$ denote the number of missing edges in $G$, that is, $m(G) := \binom{n}{2} - e(G)$. Then proving Theorem~\ref{bandwidththm} is equivalent to proving the following bound on $m(G)$:

 \begin{align}\label{thmbound}
                            m(G) \geq   \begin{cases}
                                            t(n-1) - \binom{t}{2} - \varepsilon n^2\ \ \mbox{if} \ \ t \leq \frac{n}{5} \\[0.5em]
                                           \binom{\lceil\frac{n+t}{2}\rceil + 1}{2} - \varepsilon n^2 \ \ \mbox{if} \ \ t > \frac{n}{5}.
                                        \end{cases}
\end{align}

Let $G' := G*K_t$ and label the vertices of $G'$ as $v_1, \ldots, v_{n+t}$ such that $d_{G'}(v_i) := d_i$ is the degree of vertex $v_i$ and $d_1 \leq d_2 \leq \ldots \leq d_{n+t}$.
Since $G'$ does not contain $H$ as a subgraph, it does not satisfy the degree sequence condition of Theorem~\ref{knoxthm}.
Moreover, $\delta(G') \geq t$, hence there must exist $t - \gamma(n+t) < i \leq \lceil(n+t)/2\rceil - 1$ such that $d_i < i + \gamma(n+t)$. 
From $d_1 \leq \ldots \leq d_i < i + \gamma(n+t) $ we deduce that the number of edges missing from $G'$ is at least 

\begin{align}
    m(G')    \geq \sum_{j=1}^i(n+t-1-d_j) - \binom{i}{2} & > i(n+t-1-i-\gamma(n+t)) - \binom{i}{2}\nonumber \\ 
            & \geq i((1 - 2\gamma)(n + t) - i) - \binom{i}{2} =: f(i).\label{mg'eq}
\end{align}

Set $u := \lceil(n+t)/2\rceil - 1$. Now $f(i)$ is a quadractic in $i$ and $\frac{d^2(f(i))}{di^2} < 0$. Also, note that $m(G) = m(G')$. Hence, as $t - \gamma(n+t) < i \leq u$, we have from \eqref{mg'eq} that
\begin{equation}\label{mgbound}
    m(G) \geq \min\{f(t - \gamma(n+t) ), f(u)\}.
\end{equation}
One can calculate that 
\begin{equation}\label{fineq}
    f(t - \gamma(n+t) ) \leq f(u) \ \ \mbox{if and only if} \ \ t \leq \begin{cases}
    \frac{n - 2\gamma n + 8}{5 + 2\gamma} \ \ \mbox{if $n + t$ is even}  \\[0.5em]
    \frac{n - 2\gamma n + 5}{5 + 2\gamma} \ \ \mbox{if $n + t$ is odd}.
    \end{cases}
\end{equation} 
As  $\frac{1}{n} \ll \gamma \ll \varepsilon$ we have

\begin{equation}\label{bandworkingeq1}
    f(t - \gamma(n+t)) \geq t(n-1) - \binom{t}{2} - \varepsilon n^2
\end{equation}
and
\begin{equation}\label{bandworkingeq2}
    f(u) \geq \binom{\lceil\frac{n+t}{2}\rceil + 1}{2} - \frac{\varepsilon n^2}{2}.
\end{equation} 
Moreover, for $\frac{n - 2\gamma n + 5}{5 + 2\gamma} \leq t \leq \frac{n}{5}$
we have 
\begin{equation}\label{bandworkingeq3}
    f(u) \geq \binom{\lceil\frac{n+t}{2}\rceil + 1}{2} - \frac{\varepsilon n^2}{2} \geq t(n-1) - \binom{t}{2} - \varepsilon n^2.
\end{equation}
Regardless of the parity of $n+t$, using \eqref{mgbound}--\eqref{bandworkingeq3} we conclude that \eqref{thmbound} holds.  
\end{proof}

\section{Concluding remarks}\label{conclusionsec}
In this paper we resolved the deficiency problem for $K_r$-factors. For a general fixed graph $H$, it would be interesting to prove deficiency results regarding $H$-factors. As a starting point for this problem we pose the following question.
Let  $\alpha(H)$ denote the size of the largest independent set in $H$. 

\begin{question}\label{q}
Let $K := K_n$ and $A \subseteq K$ such that $A = K_{\frac{\alpha(H)(n+t)}{|H|} + 1}$. Define $EX_H(n,t)$ to be the graph obtained by removing $E(A)$ from $K$.
Does there exist a constant $c:=c(H)>0$ such that if $t\geq c n$ and $G$ is an $n$-vertex graph so that
$G*K_t$ does not contain an $H$-factor then $e(G) \leq e(EX_H(n,t))+o(n^2)$? 
\end{question}
Note that Theorem~\ref{bandwidththm} answers this question in the affirmative e.g. for $H=K_{s,s}$ (for fixed $s \in \mathbb N$).
On the other hand, at least for some $H$ one cannot remove the $o(n^2)$
term in Question~\ref{q} completely. Indeed, let $H =K_{1,s}$ where $s \geq 2$ and consider the $n$-vertex graph $EX'_H(n,t)$ obtained from $EX_H(n,t)$
by adding a maximal matching in $V(A)$. It is easy to see that
$EX'_H(n,t)*K_t$ does not contain a $K_{1,s}$-factor. 
This example suggests it might be rather challenging to resolve the $H$-factor deficiency problem completely for all graphs $H$.

It would also be interesting to prove bandwidth deficiency results in the vein of Theorem~\ref{bandwidththm} for non-bipartite graphs $H$.

\section*{Acknowledgment}
The authors are grateful to the referee for a helpful and careful review.

\medskip

{\footnotesize \obeylines \parindent=0pt
\begin{tabular}{lll}
	Andrea Freschi, Joseph Hyde \& Andrew Treglown \\
	School of Mathematics 						\\
	University of Birmingham					\\
	Birmingham											 \\
	B15 2TT														 \\
	UK															
\end{tabular}
}
\begin{flushleft}
{\it{E-mail addresses}:
\tt{ $\{$axf079, jfh337, a.c.treglown$\}$@bham.ac.uk}}
\end{flushleft}

\end{document}